\documentclass{amsart}

\usepackage{amsfonts}
\usepackage{color}
\usepackage{graphicx}
\usepackage[bookmarksnumbered,colorlinks, linkcolor=blue, citecolor=red, pagebackref, bookmarks, breaklinks]{hyperref}

\newtheorem{thm}{Theorem}[section]
\newtheorem{cor}[thm]{Corollary}
\newtheorem{lema}[thm]{Lemma}
\newtheorem{prop}[thm]{Proposition}
\theoremstyle{definition}

\theoremstyle{remark}

\newtheorem{rem}[thm]{Remark}
\numberwithin{equation}{section}
\newcommand{\R}{\mathbb R}
\newcommand{\N}{\mathbb N}

\newcommand{\Sn}{\mathbb{S}^{n-1}}

\newcommand{\A}{\mathcal{A}}

\newcommand{\LL}{\mathcal{L}}

\newcommand{\Hn}{\mathcal{H}^{n-1}}

\newcommand{\ve}{\varepsilon}

\def\diver{\mathop{\text{\normalfont div}}}

\begin{document}

\title{Asymptotic behavior for anisotropic fractional energies}
	
\author[J. Fern\'andez Bonder]{Julian Fern\'andez Bonder}

\address[J. Fern\'andez Bonder and A, Salort]{Instituto de C\'alculo (UBA - CONICET) and \hfill\break\indent Departamento  de Matem\'atica, FCEyN, Universidad de Buenos Aires, \hfill\break\indent Pabell\'on I, Ciudad Universitaria (1428), Buenos Aires, Argentina.}

\email[J. Fern\'andez Bonder]{{\tt jfbonder@dm.uba.ar}}
\email[A. Salort]{{\tt asalort@dm.uba.ar}}

\author[A. Salort]{Ariel Salort}

\subjclass[2020]{35J92, 35R11, 35B27}


\keywords{Fractional energies, fractional order Sobolev spaces, homogenization}

\thanks{This workd was partially supported by UBACYT Prog. 2018 20020170100445BA and by ANPCyT PICT 2019-00985. J. Fern\'andez Bonder and A. Salort are members of CONICET}

\begin{abstract}
In this paper we investigate the asymptotic behavior of anisotropic fractional energies as the fractional parameter $s\in (0,1)$ approaches both $s\uparrow 1$ and $s\downarrow 0$ in the spirit of the celebrated papers of Bourgain-Brezis-Mironescu \cite{BBM} and Maz'ya-Shaposhnikova \cite{MS}.

Then, focusing con the case $s\uparrow 1$ we analyze the behavior of solutions to the corresponding minimization problems and finally, we also study the problem where a homogenization effect is combined with the localization phenomena that occurs when $s\uparrow 1$.
\end{abstract}

\maketitle

\section{Introduction}

The celebrated result by Bourgain, Brezis and Mironescu establishes the behavior of the so-called Gagliardo seminorm in fractional order Sobolev spaces of order $s$ as $s\uparrow 1$, providing new characterizations for functions in the Sobolev space $W^{1,p}(\Omega)$.  More precisely, given a smooth bounded domain $\Omega\subset\R^n$, $n\geq 1$ and $p\in [1,\infty)$, for any $u\in W^{1,p}(\Omega)$ in \cite{BBM} it is  proved that
$$
\lim_{s\uparrow 1} (1-s) \iint_{\Omega \times \Omega} \frac{|u(x)-u(y)|^p}{|x-y|^{n+sp}}\,dxdy = \mathcal{K}_{p,n} \|\nabla u\|_p^p
$$
where the constant $\mathcal{K}_{p,n}$ is given by
$$
 \mathcal{K}_{p,n} = \frac1p \int_{\Sn} |\omega_1|^p \, d\Hn.
$$
Here $\Sn$ is the unit sphere in $\R^n$ and $\Hn$ the Hausdorff $(n-1)-$dimensional measure.

The formula above was proved to hold  with less assumptions on the domain. In fact, in \cite{LS} it is established the validity of the BBM-formula for any open set $\Omega\subset \R^n$, and recently,  in \cite{DD},  for any bounded domain. This analysis was completed in \cite{D,P}, where it was proven that a similar formula holds for functions of bounded variation when $\Omega\subset\R^n$ is a bounded Lipschitz set.

Motivated with the results in \cite{BBM}, Maz'ya and Shaposhnikova complemented the study by analyzing the behavior of the seminorm as $s\downarrow 0$. In fact, the authors proved in \cite{MS} that for any $n\geq 1$ and $p\in [1,\infty)$
$$
\lim_{s\downarrow 0} s \iint_{\R^n\times \R^n} \frac{|u(x)-u(y)|^p}{|x-y|^{n+sp}}\,dxdy =  \mathcal{C}_{p,n}  \|u\|_p^p
$$
whenever $u\in D^{s,p}(\R^n)$ for some $s\in (0,1)$ where $D^{s,p}(\R^n)$ is the completion of $C^\infty_c(\R^n)$ with respect to the Gagliardo seminorm. The constant $\mathcal{C}_{p,n}$ is given by
$$
\mathcal{C}_{p,n} = \frac{4 \pi^\frac{n}{2}}{p\Gamma(\tfrac{n}{2})},
$$
where $\Gamma$ denotes the Gamma function.

\medskip

The singular limits mentioned above are natural and have a physical relevance in the framework of the theory of L\'evy processes. This has led to the fact that in the last years, a huge effort in trying to extend the asymptotic results as $s\uparrow  1$ and $s\downarrow 0$ proved in \cite{BBM, MS} to different contexts has been carried out. We mention just some examples: for the theory of fractional $s-$perimeters, the analysis of the asymptotic limits was addressed in \cite{ADM, DFPV}; the extension to functions allowing a behavior more general than a power was done in \cite{ACPS2, ACPS, CMSV, FBS} in the context of fractional Orlicz-Sobolev spaces; in the magnetic setting, the behavior of the corresponding seminorms was studied \cite{NPSV, PSV}; the extension of magnetic fractional Orlicz-Sobolev spaces was dealt in \cite{FBSmag, MSV}. 

\medskip

The purpose of this paper is to study the asymptotic behavior as $s\uparrow 1$ and $s\downarrow 0$ of anisotropic Gagliardo seminorms, that is, the quantity
$$
J_{m,s}(u):=\frac{1-s}{p}\iint_{\R^n\times\R^n} m(x,x-y) \frac{|u(x)-u(y)|^p}{|x-y|^{n+sp}}\, dxdy,
$$
where $m$ is a function bounded away from 0 and infinity satisfying some suitable conditions (see hypotheses \eqref{H1}--\eqref{H3} below).

Let us conclude this section by describing our main results. In Theorem \ref{teo.bbm} we prove that given $u\in L^p(\R^n)$ fixed,
$$
\lim_{s\uparrow 1} J_{m,s}(u) = \int_{\R^n} \A(x, \nabla u)\,dx
$$
where  
$$
\A(x,\xi)= \frac1p \int_{\Sn} a(x,w)|\xi\cdot w|^p  \, d\Hn
$$
and $a(x,\omega)$ is a {\em radial limit} of the weight function $m$ (see \eqref{H3}).

In Theorem \ref{main2} we also treat the case of a sequence, i.e., the behavior of $J_{m,s}(u_s)$ as $s\uparrow 1$, where $\{u_s\}_s$ is a sequence of functions in $L^p(\R^n)$ such that $(1-s)[u_s]_{s,p}^p + \|u_s\|_p^p$ is uniformly bounded.

When $u\in W^{s_0,p}(\R^n)$ for some $s_0\in (0,1)$, then in Theorem \ref{MS} we prove that 
$$
\lim_{s\downarrow 0} sJ_{m,s}(u) = \int_{\R^n} |u|^p b(x)\, dx,
$$
where $b(x)=\lim_{s\downarrow 0} b_s(x)$  a.e. $x\in\R^n$, and
$$
b_s(x) := 2s\int_{\Sn} \int_{2|x|}^\infty \frac{m(x,r\omega) }{r^{sp+1}} \,dr\,d\Hn.
$$

In the last part of the paper we analyze whether homogenization and localization processes can be interchanged. To be more precise, observe that the functional $J_{m,s}$ of a function $u_s\in W^{s,p}_0(\Omega)$, $\Omega\subset \R^n$, is related with weak solutions to
$$
\LL_{m,s} u_s = f \quad  \text{ in }\Omega
$$
where $f\in L^{p'}(\Omega)$ and $\LL_{m,s}$ is the Fr\'echet derivative of $J_{m.s}$, i.e., 
$$
\LL_{m,s}(u)= p.v. (1-s)\int_{\R^n} m(x,x-y)\frac{|u(x)-u(y)|^{p-2}(u(x)-u(y))}{|x-y|^{n+sp}}\,dy.
$$
In Section \ref{S6} we consider solutions of a family of kernels $m_\ve(x,x-y)$, $\ve>0$ having the form $m_\ve(x,x-y)=m(\tfrac{x}{\ve},x-y)$, where $m(x,x-y)$ fullfills the previous assumptions and it is further a $Q-$periodic function in the first variable, being $Q$ the unit cube in $\R^n$. Given a solution of $u_{s,\ve}\in W^{1,p}_0(\Omega)$ to
$\LL_{m_\ve,s} u_{\ve,s} = f$ in $\Omega$, in Proposition \ref{lim.ve.s} we prove that
$$
\lim_{\ve\downarrow 0} \left( \lim_{s\uparrow 1} u_{s,\ve} \right)= u^*,
$$
(in the $L^p(\Omega)$ sense) where $u^*\in W^{1,p}_0(\Omega)$ is the weak solution of
$$
-\diver(\nabla_\xi \A^*(\nabla u^*)) = f \quad  \text{in }\Omega \qquad \text{ with } \A^*(\xi)= \inf_{v\in W^{1,1}_{per}(Q)} \int_Q \A(y, \xi+\nabla v(y))\, dy.
$$

On the other hand, in Propisition \ref{lim.s.ve} we get that
$$
\lim_{s\uparrow 1} \left(  \lim_{\ve\downarrow 0}  u_{s,\ve} \right) = \bar u
$$
(in the $L^p(\Omega)$ sense), where $\bar u\in W^{1,p}_0(\Omega)$ is the solution to
$$
-\diver(\nabla_\xi \bar\A(\nabla \bar u))= f \quad  \text{in }\Omega, \qquad \text{ with } \bar \A(\xi) = \int_Q \A(y,\xi)\, dy.
$$
This shows that in general, homogenization and localization do not commute.

\subsection*{Organization of the paper}
After this introduction, in Section \ref{S2}, we collect some preliminaries, and establish some notation that will be used in the sequel.

In Section \ref{S3}, we analyze the problems for $s\uparrow 1$, the so-called BBM-type results in the spirit of Bourgains-Brezis-Mironescu's paper \cite{BBM}.

In Section \ref{S4}, we analyze the problem $s\downarrow 0$, the MS-type results, in the spirit of Maz'ya-Shaposhnikova's paper \cite{MS}.

In Section \ref{S5} we connect the BBM-type results of Section \ref{S3} with the asymptotic behavior of solutions to nonlocal problems and the transition to solutions to local ones.

Finally, in Section \ref{S6}, we investigate the interplay between localization (i.e. $s\uparrow 1$) and homogenization.

\section{Preliminaries}\label{S2}

\subsection{Fractional order Sobolev spaces}
Throughout this article we will use the standard Gagliardo definition of fractional order Sobolev spaces. That is:
Given a fractional parameter $s\in (0,1)$ and an integrable parameter $p\in [1,\infty)$, the fractional order Sobolev space, $W^{s,p}(\R^n)$ is defined as
$$
W^{s,p}(\R^n) := \left\{ u\in L^p(\R^n)\colon [u]_{s,p}^p <\infty\right\},
$$
where $[\, \cdot\, ]_{s,p}$ is the so-called Gagliardo seminorm that is defined as
$$
[u]_{s,p}^p:= \iint_{\R^n\times\R^n} \frac{|u(x)-u(x-h)|^p}{|h|^{n+sp}}\, dxdh.
$$

This space is endowed with the norm
$$
\|u\|_{s,p} = \left( \|u\|_p^p + [u]_{s,p}^p\right)^\frac1p
$$
and $(W^{s,p}(\R^n), \|\cdot\|_{s,p})$ is a separable Banach space, that is reflexive if $p>1$.

When considering domains $\Omega\subset\R^n$ we will use the notation $W^{s,p}_0(\Omega)$ to denote the set of functions in $W^{s,p}(\R^n)$ that vanishes outside $\Omega$, namely
$$
W^{s,p}_0(\Omega) := \{u\in W^{s,p}(\R^n)\colon u=0 \text{ a.e. in } \R^n\setminus \Omega\}.
$$

Observe that this space agrees with the closure of test functions in $\Omega$ if, for instance, $\Omega$ has Lipschitz boundary or if $s<\tfrac1p$.

For these spaces, the Rellich-Kondrashov compactness result holds true, i.e.
\begin{thm}\label{R-K}
Assume that $s\in (0,1)$ and $p\in [1,\infty)$ and let $\{u_k\}_{k\in\N}\subset W^{s,p}_0(\Omega)$ be a bounded sequence. Then, there exists $u\in W^{s,p}_0(\Omega)$ and a subsequence $\{u_{k_j}\}_{j\in\N}\subset \{u_k\}_{k\in\N}$ such that
$$
u_{k_j}\to u \quad \text{ in } L^p_\text{loc}(\Omega).
$$
If $\Omega$ is bounded, the convergence is in $L^p(\Omega)$.
\end{thm}

All of the above mentioned results are well known and can be found, for instance, in \cite{DPV}.

\subsection{Some notation}

In several places of the paper, the following notation will be used:
\begin{itemize}
\item The unit sphere in $\R^n$ will be denoted by $\Sn$.
\item The $(n-1)-$dimensional Hausdorff measure will be denoted by $\Hn$.
\item The volume of the unit ball in $\R^n$ will be denoted by $\omega_n$.
\item The volume of the $(n-1)-$dimensional unit sphere in $\R^n$ is then $n\omega_n$.
\end{itemize}

\subsection{Anisotropic fractional energies}
We consider a kernel function $m=m(x,h)$, $m\in L^\infty(\R^n\times\R^n)$  and for each function $m$ and each fractional parameter $s\in (0,1)$ we define the functional $J_{m,s}\colon W^{s,p}_0(\Omega)\to \R$,
$$
J_{m,s}(u) := \frac{(1-s)}{p}\iint_{\R^n\times\R^n} m(x,h) \frac{|u(x)-u(x-h)|^p}{|h|^{n+sp}}\, dxdh,
$$
where $\Omega$ is a domain in $\R^n$, not necessarily bounded.

If the kernel function $m$ is bounded below away from 0, the functional $J_{m,s}$ is coercive, so we impose the following condition on $m$:
\begin{equation}\label{H1}\tag{$H_1$}
m_- \le m(x,h)\le m_+
\end{equation}
for some $0<m_-\le  m_+<\infty$.

It is easy to see that $J_{m,s}$ is Fr\'echet differentiable. If we try to obtain an integral representation of the derivative $J_{m,s}'(u)\in W^{-s,p'}(\Omega)$, we need to impose some symmetry assumptions on the kernel $m$, namely,
\begin{equation}\label{H2}\tag{$H_2$}
m(x,h)=m(x-h,-h).
\end{equation}
Under this condition, it is easy to see (see for instance \cite{FBRS}), that the derivative $J_{m,s}'(u)$ has the following integral representation,
\begin{equation}\label{Lm}
\begin{split}
J'_{m,s}(u) := &\LL_{m,s}(u)\\
= &p.v. (1-s) \int_{\R^n} m(x,h)\frac{|u(x)-u(x-h)|^{p-2}(u(x)-u(x-h))}{|h|^{n+sp}}\,dh,
\end{split}
\end{equation}
where $p.v.$ stands for {\em in principal value}.

Observe that hypotheses \eqref{H2} is by no means restrictive, since denoting
$$
m_\text{sym}(x,h) = \frac{m(x,h) + m(x-h,-h)}{2}
$$
we have that $m_\text{sym}$ satisfies \eqref{H2} and 
$$
J_{m,s} = J_{m_\text{sym},s}.
$$

In order to analyze the case where $s\uparrow 1$ in our functionals $J_{m,s}$ we need to assume some asymptotic behavior on the kernel $m$. This condition, though it seems quite technical right now it will become apparent later on:

There exists a function $a:\R^n\times \Sn\to \R$ such that
\begin{equation}\label{H3}\tag{$H_3$}
m(x,r\omega) = a(x,\omega) + O(r)
\end{equation}
uniformly in $\omega\in \Sn$.

This condition is saying that $m$ has some singular behavior on the diagonal that is determined by the angle in which one approaches the origin.

Observe that $m\in C^1$ implies that the limit function $a$ in \eqref{H3} is independent of $\omega$. In fact, $a(x,\omega)=m(x,0)$ in this case and a typical nontrivial example to keep in mind is the following 
$$
m(x,h) = \left|M(x,h) \frac{h}{|h|}\right|^\alpha,\ \alpha\neq 0
$$
where $M(x,h)\in \R^{n\times n}$ is a symmetric uniformly elliptic matrix with the structural hypothesis
$$
M(x,h)=M(x-h,-h).
$$
In this case, the function $a(x,\omega)$ is given by
$$
a(x,\omega) = \left|M(x,0)\omega\right|^\alpha.
$$

\section{Limit as $s\uparrow 1$ of $J_{m,s}$}\label{S3}
The purpose of this section is to analyze the behavior as $s\uparrow 1$ of the functional $J_{m,s}$. This is the extension of the celebrated result of Bourgain-Brezis-Mironescu to the anisotropic case.

First we begin by studying the pointwise limit of the funcionales that is much simpler. Later on, we will deal with the Gamma-convergence of the funcional that is more subtle.

\subsection{Pointwise limit}
The results in this subsection are inspired by \cite{AV} where the authors consider some particular case of weight function for $p=2$. 

To begin with, we cite a Lemma that can be found in \cite{BBM}.

\begin{lema} \label{lema.cota}
Given $u\in W^{s,p}_0(\Omega)$ it holds that
$$
\int_{\R^n} m(x,h) \frac{|u(x)-u(x-h)|^p}{|h|^{n+sp}}\,dh \le m_+ [u]_{s,p}^p \leq  \frac{n\omega_n m_+}{p} \left(\frac{1}{1-s} \|\nabla u\|_p^p + \frac{2^p}{s} \|u\|_p^p\right).
$$
\end{lema}

\begin{proof}
Just combine \eqref{H1} with \cite[Theorem 1]{BBM}.
\end{proof}

The following proposition is key in the proof of our main result.

\begin{prop} \label{prop.clave}
Given $u\in C^2_c(\R^n)$ and a fixed $x\in\R^n$ we have that
$$
\lim_{s\uparrow 1} (1-s) \int_{\R^n} m(x,h) \frac{|u(x)-u(x-h)|^p}{|h|^{n+sp}}\,dh = \A(x, \nabla u)
$$
where $\A(x,\xi)$ is given by 
\begin{equation} \label{coef}
\A(x,\xi)= \frac1p \int_{\Sn} a(x,w)|\xi\cdot w|^p  \, d\Hn.
\end{equation}
\end{prop}

\begin{rem}
In the linear case, that is when $p=2$, the operator $\A(x,\xi)$ has a more explicit form,
$$
\A(x,\xi)=A(x)\xi\cdot\xi,
$$
where the matrix $A\in \R^{n\times n}$ is given by
$$
a_{ij}(x) = \frac12 \int_{\Sn} w_iw_j a(x,w)\, d\Hn.
$$
\end{rem}

\begin{proof}[Proof of Proposition \ref{prop.clave}]
For each fixed $x\in\R^n$ we split the integral
\begin{align*}
\int_{\R^n} m(x,h) \frac{|u(x)-u(x-h)|^p}{|h|^{n+sp}}\,dh  &= 
\left(\int_{|h|\geq 1} + \int_{|h|< 1} \right) m(x,h) \frac{|u(x)-u(x-h)|^p}{|h|^{n+sp}}\,dh \\
&= I_1 + I_2.
\end{align*}
Since \eqref{H1} holds, 
$$
|I_1|\leq 2^{p-1} m_+ \|u\|_\infty^p  \int_{|h|\geq 1} \frac{1}{|h|^{n+sp}}\,dh <\infty,
$$
and we focus only on $I_2$. Since $x\mapsto |x|^p$ is locally Lipschitz and $u\in C^2$, we have that
$$
\left| \frac{|u(x)-u(x-h)|^p}{|h|^{sp}} - \frac{|\nabla u(x)\cdot h|^p}{|h|^{sp}} \right|\leq L \frac{|u(x)-u(x-h) - \nabla u(x)\cdot h|}{|h|^{s}}  \le C |h|^{2-s}
$$
where $C$ depends of the $C^2-$norm of $u$.

Since the following integral vanishes
$$
\lim_{s\uparrow 1}(1-s)\int_{|h|\leq 1} |h|^{2-s-n}\,dh = \lim_{s\uparrow 1}(1-s)\frac{n\omega_n}{2-s} =0,
$$
it follows that
\begin{align*}
\lim_{s\uparrow 1} (1-s)I_2 &=\lim_{s\uparrow 1} (1-s)\int_{|h|\leq 1}m(x,h) \frac{|\nabla u(x)\cdot h|^p}{|h|^{n+sp}}\,dh
\end{align*}
Hence, by using polar coordinates we get
\begin{align*}
\int_{|h|\leq 1}m(x,h) \frac{|\nabla u(x)\cdot h|^p}{|h|^{n+sp}}\,dh 
&=  \int_{|h|\leq 1}m(x,h) \frac{|\nabla u(x)\cdot \tfrac{h}{|h|}|^p}{|h|^{n+sp-p}}\,dh \\
&=  \int_0^1 \int_{\Sn} m(x,r\omega) |\nabla u(x)\cdot \omega|^p r^{p(1-s)-1} \, d\Hn\, dr\\
&=  \int_{\Sn} |\nabla u(x)\cdot \omega|^p \left( \int_0^1 m(x,r\omega)  r^{p(1-s)-1} \,dr\right) \, d\Hn.
\end{align*}

Observe that from \eqref{H3} we have that $m(x,r\omega)= a(x,\omega)+O(r)$,  which implies that
\begin{align*}
\int_0^1 m(x,r\omega)  r^{p(1-s)-1} \,dr &=  \int_0^1 a(x,\omega)  r^{p(1-s)-1} \,dr  +  \int_0^1   O(r^{p(1-s)}) \,dr\\
&=  a(x,\omega) \frac{1}{p(1-s)}  +  O(1)
\end{align*}
and consequently
$$
\lim_{s\uparrow 1}(1-s)\int_0^1 m(x,r\omega)  r^{p(1-s)-1} \,dr = \frac1p a(x,\omega).
$$ 
Finally, 
\begin{align*}
\lim_{s\uparrow 1}(1-s)& \int_{|h|\leq 1}m(x,h) \frac{|\nabla u(x)\cdot h|^p}{|h|^{n+sp}}\,dh\\
&= \lim_{s\uparrow 1}(1-s) \int_{\Sn} |\nabla u(x)\cdot \omega|^p \left( \int_0^1 m(x,r\omega)  r^{p(1-s)-1} \,dr\right) \, d\Hn\\
&=  \frac1p \int_{\Sn} a(x,\omega)|\nabla u(x)\cdot \omega|^p  \, d\Hn,
\end{align*}
which concludes the proof.
\end{proof}

We are ready to state and proof our main result in this subsection.
 
\begin{thm} \label{teo.bbm}
Given $u\in L^p(\R^n)$ and a fixed $x\in\R^n$ we have that
$$
\lim_{s\uparrow 1} (1-s) \iint_{\R^n\times \R^n} m(x,h) \frac{|u(x)-u(x-h)|^p}{|h|^{n+sp}}\,dxdh = \int_{\R^n} \A(x, \nabla u)\,dx
$$
where $\A(x,\xi)$ is given in \eqref{coef}.
\end{thm}

\begin{proof}
Given $u\in C^2_c(\R^n)$ with $\text{supp}(u)\subset B_R(0)$, in view of Proposition \ref{prop.clave}, it only remains to show the existence of an integrable majorant for $(1-s)F_s$, where $F_s$ is given by
$$
F_s(x):=\int_{\R^n} m(x,h) \frac{|u(x)-u(x-h)|^p}{|h|^{n+sp}}\,dh.
$$
But, thanks to \eqref{H1}, this task is the same done in \cite{BBM}, more precisely, 
$$
(1-s) |F_s(x)| \leq C m_+ \left( \chi_{B_R(0)}(x) + |x|^{-(n+\tfrac12)}\chi_{B_r(0)^c}(x) \right) \in L^1(\R^n).
$$
Then, from Proposition \ref{prop.clave} and the Dominated Convergence Theorem the result follows for any $u\in C^2_c(\R^n)$.

Using Lemma \ref{lema.cota} and \cite[Theorem 2]{BBM}, the result is extended to an arbitrary function $u\in W^{1,p}(\R^n)$.

Finally, arguing as in \cite[Theorem 2]{BBM} (see also \cite{FBS}) it holds that if
$$
\liminf_{s\uparrow 1} (1-s) \iint_{\R^n\times \R^n} m(x,h) \frac{|u(x)-u(x-h)|^p}{|h|^{n+sp}}\,dxdh  <\infty,
$$
then $u\in W^{1,p}(\R^n)$ and the result follows.
\end{proof}

\subsection{The case of a sequence}
In this subsection we deal with the case of a sequence that will imply, among other things, the Gamma convergence of the functionals $J_{m,s}$.

\begin{thm}\label{main2}
Let $0\le s_k\  \uparrow 1$ and $\{u_k\}_{k\in\N}\subset L^p(\R^n)$ be such that
$$
\sup_{k\in\N} (1-s_k) [u_k]_{s_k,p}^p <\infty \quad \text{and}\quad \sup_{k\in\N} \|u_k\|_{L^p(\R^n)} < \infty.
$$
Then there exists $u\in L^p(\R^n)$ and a subsequence $\{u_{k_j}\}_{j\in\N}\subset \{u_k\}_{k\in\N}$ such that $u_{k_j}\to u$ in $L^p_{loc}(\R^n)$. Moreover $u\in W^{1,p}(\R^n)$ and the following estimate holds
$$
\int_{\R^n} \A(x, \nabla u)\, dx \le \liminf_{k\to\infty} J_{m,s_k}(u_k).
$$
\end{thm} 

The proof of the above result will be a direct consequence of the following useful estimate:

\begin{thm} \label{teo.s1.s2}
Let $0<s_1<s_2<1$ and $u \in L^p(\R^n)$. Then
\begin{align*}
J_{m,s_1}(u)\le 2^{p(1-s_1)} J_{m,s_2}(u) + \frac{2^{p-1}m_+ n\omega_n (1-s_1)}{s_1}\|u\|_p^p.
\end{align*}
\end{thm}

The key point in proving Theorem \ref{teo.s1.s2} is the following lemma that is proved in \cite{BBM}. \begin{lema}[Lemma 2, \cite{BBM}] \label{lema.bbm}
Let $g,h:(0,1) \to \R^+$ measurable functions. Suppose that for some constant $c>0$ it holds that $g(t) \leq c g(\tfrac{t}{2})$ for $t\in (0,1)$ and that $h$ is decreasing. Then, given $r>-1$, 
$$
\int_0^1 t^r g(t)h(t)\,dt \geq \frac{r+1}{2^{r+1}}\int_0^1 t^r g(t)\,dt \int_0^1 t^r h(t) \,dt.
$$
\end{lema}
Actually, the proof in \cite{BBM} is done with $c=1$. The extension for general $c>0$ is immediate.

Now we proceed with the proof of the estimate.

\begin{proof}[Proof of Theorem \ref{teo.s1.s2}]
The proof is very similar to that of \cite[Theorem 4]{BBM} (see also \cite[Theorem 5.1]{FBS}). We include some details in order to make the paper self contained.

Given $u\in L^p(\R^n)$, we define for $t>0$ and $0<s<1$,
\begin{align*}
F(t)&=\int_{\Sn} \int_{\R^n}  m(x,tw) |u(x)-u(x-tw)|^p \,dx \,d\Hn\\
&=\frac{1}{t^{n-1}}\int_{|h|=t} \int_{\R^n} m(x,h) |u(x)-u(x-h)|^p \,dx \,d\Hn
\end{align*}
and $g(t) = \frac{F(t)}{t^p}$.

From \cite[p. 13]{BBM} and assumption \eqref{H1} if follows that
$$
g(2t)\leq \frac{m_+}{m_-} g(t).
$$
Then, observe that
\begin{align} \label{dess1}
\begin{split}
\int_{|h|<1} \int_{\R^n}  &m(x,h) \frac{|u(x)-u(x-h)|^p}{|h|^{n+sp}} dx\,dh\\
&=\int_0^1 \int_{|h|=t} \int_{\R^n}  m(x,h)\frac{|u(x)-u(x-h)|^p}{t^{n+sp}} dx\, d\Hn \,dt\\
&=\int_0^1 \frac{F(t)}{t^{1+sp}} \,dt = \int_0^1 \frac{g(t)}{t^{1-p(1-s)}}\,dt.
\end{split}
\end{align}
Consider now $0<s_1<s_2<1$. Therefore,  
\begin{align*}
\int_0^1 \frac{1}{t^{1-p(1-s_2)}} g(t)\, dt = \int_0^1 \frac{1}{t^{1-p(1-s_1)}} g(t) \frac{1}{t^{p(s_2-s_1)}}\, dt.
\end{align*}
Now, from Lemma \ref{lema.bbm} with $r=p(1-s_1)-1$ and $h(t) = t^{-p(s_2-s_1)}$ we get
\begin{equation}\label{desigualdad.fundamental}
\begin{split}
\int_0^1 \frac{1}{t^{1-p(1-s_2)}} g(t)\, dt \ge& \frac{p(1-s_1)}{2^{p(1-s_1)}} \int_0^1 \frac{1}{t^{1-p(1-s_1)}} g(t)\, dt \int_0^1 \frac{1}{t^{1-p(1-s_2)}}\, dt\\
=& \frac{1}{2^{p(1-s_1)}} \frac{1-s_1}{1-s_2} \int_0^1 \frac{1}{t^{1-p(1-s_1)}} g(t)\, dt.
\end{split}
\end{equation}
From \eqref{dess1} and \eqref{desigualdad.fundamental} we deduce that
\begin{align}\label{primera.parte}
\begin{split}
\frac{(1-s_1)}{2^{p(1-s_1)}}&\int_{|h|<1}\int_{\R^n} m(x,h)\frac{|u(x)-u(x-h)|^p}{|h|^{n+s_1 p}}\, dxdh\\
&\le (1-s_2)\int_{\{|h|<1\}}\int_{\R^n} m(x,h)\frac{|u(x)-u(x-h)|^p}{|h|^{n+s_2 p}}\, dxdh.
\end{split}
\end{align}

Finally we observe that
\begin{align*}
\int_{\{|h|\ge 1\}} \int_{\R^n} m(x,h)\frac{|u(x)-u(x-h)|^p}{|h|^{n+sp}}\, dxdh &\le 2^p m_+ n\omega_n \|u\|_p^p \int_1^\infty \frac{1}{t^{1+sp}}\, dt\\
&= \frac{2^p m_+ n\omega_n}{sp} \|u\|_p^p.
\end{align*}
The proof concludes combining this last inequality with \eqref{primera.parte}. 
\end{proof}
 
Now we can proceed with the proof of Theorem \ref{main2}.

\begin{proof}[Proof of Theorem \ref{main2}]
With the help of Theorem \ref{teo.s1.s2} the proof is the same as \cite[Theorem 4]{BBM} and \cite[Theorem 5.1]{FBS}

We include some details for the reader's convenience.

Let $0<s_k\uparrow 1$ and $\{u_k\}_{k\in\N}\subset L^p(\R^n)$ such that
$$
\sup_{k\in\N} (1-s_k)[u_k]_{s_k,p}^p <\infty \quad \text{and}\quad \sup_{k\in\N} \|u_k\|_p<\infty.
$$

For a fixed $0<t<1$ using Theorem \ref{teo.s1.s2} we have that $\{u_k\}_{k\in\N}\subset W^{t,p}(\R^n)$ is bounded and so by the Rellich-Kondrashov compactness Theorem (Theorem \ref{R-K}), there exists a subsequence (still denoted by $\{u_k\}_{k\in\N}$) and a limit function $u\in L^p(\R^n)$ such that $u_k\to u$ in $L^p_\text{loc}(\R^n)$. We can also assume that $u_k\to u$ a.e. in $\R^n$.

Now, by Fatou's Lemma, we have
\begin{align*}
\iint_{\R^n\times\R^n} m(x,h) &\frac{|u(x)-u(x-h)|^p}{|h|^{n+tp}}\, dxdy\\
& \le \liminf_{k\to\infty} \iint_{\R^n\times\R^n} m(x,h)\frac{|u_k(x)-u_k(x-h)|^p}{|h|^{n+tp}}\, dxdy
\end{align*}
and by Theorem \ref{teo.s1.s2}  we obtain
\begin{align*}
\frac{1-t}{2^{(1-t)p}} \iint_{\R^n\times\R^n} &m(x,h)\frac{|u(x)-u(x-h)|^p}{|h|^{n+tp}}\, dxdy \\
\le& \liminf_{k\to\infty} (1-s_k)\iint_{\R^n\times\R^n} m(x,h)\frac{|u_k(x)-u_k(x-h)|^p}{|h|^{n+s_kp}}\, dxdy\\
&+ \frac{n\omega_n 2^p(1-t)m_+}{tp} M
\end{align*}
wher $M=\sup_{k\in\N} \|u_k\|_p^p$.

Finally the result follows taking the limit $t\uparrow 1$ and using Theorem \ref{teo.bbm}.
\end{proof}

\section{Limit as $s\downarrow 0$ of $J_{m,s}$}\label{S4}

In this section we analyze the limit case where $s\downarrow 0$ of the functionals $J_{m,s}$. This is what is called a {\em Maz'ya-Shaposhnikova type result} after the results obtained in \cite{MS}. That is we are interested in studying  the limit
$$
\lim_{s\downarrow 0} sJ_{m,s}(u).
$$

First we define the following weights depending on $m$ and $s$, 
$$
b_s(x) := 2s\int_{\Sn} \int_{2|x|}^\infty \frac{m(x,r\omega) }{r^{sp+1}} \,dr\,d\Hn.
$$
Observe that this weight has the following bounds
$$
\frac{2^{1-sp}n\omega_n}{p} \frac{m_-}{|x|^{sp}}\le b_s(x)\le\frac{2^{1-sp}n\omega_n}{p} \frac{m_+}{|x|^{sp}},
$$
where $m_{\pm}$ are given in \eqref{H1}.

We will assume that there exists the limit function
$$
b(x)=\lim_{s\downarrow 0} b_s(x)\quad \text{a.e. }x\in\R^n.
$$

Our main result in the section is
\begin{thm}\label{MS}
Under the above assumptions and notations, if $u\in W^{s_0,p}(\R^n)$ for some $s_0\in (0,1)$, then
$$
\lim_{s\downarrow 0} sJ_{m,s}(u) = \int_{\R^n} |u|^p b(x)\, dx.
$$
\end{thm}

The proof of this result follows the general strategy developed in \cite{MS} but also applies some ideas from \cite{Cianchi}.

The proof will be a direct consequence of the next two lemmas. The first one is a Hardy-type inequality with weights
\begin{lema}\label{liminfMS}
Let $u\in W^{s_0, p}(\R^n)$ for some $s_0\in (0,1)$, then
$$
\liminf_{s\downarrow 0} sJ_{m,s}(u) \ge \int_{\R^n} |u|^p b(x)\, dx.
$$
\end{lema}

\begin{proof}
Let us call
$$
I^p=\int_{\R^n}\int_{|h|>2|x|} m(x,h) \frac{|u(x)|^p}{|h|^{n+sp}} \,dh dx.
$$
Then
\begin{align*}
I^p &=  \int_{\R^n}\left( \int_{|h|\geq 2|x|} \frac{m(x,h)}{|h|^{n+sp}}dh\right)|u(x)|^p\,dx \\
&=  \int_{\R^n} \left( \int_{\Sn} \int_{2|x|}^\infty \frac{m(x,r\omega)}{r^{1+sp}} \,dr d\Hn \right) |u(x)|^p \,dx.
\end{align*}
Now, for any $\ve>0$,
\begin{align*}
I^p  \leq &  (1+\ve)^{p-1} \int_{\R^n}\int_{|h|\geq 2|x|} m(x,h)\frac{|u(x)-u(x-h)|^p}{|h|^{n+sp}}\,dhdx\\
& + \left(\frac{1+\ve}{\ve}\right)^{p-1} \int_{\R^n}\int_{|h|\geq 2|x|} m(x,h)\frac{|u(x-h)|^p}{|h|^{n+sp}}\, dhdx\\
:= & (1+\ve)^{p-1} (a)^p + \left(\frac{1+\ve}{\ve}\right)^{p-1}  (b)^p.
\end{align*}

Let us first bound $(b)$. As $|h|>2|x|$ we have that $\frac23|x-h|<|h|<2|x-h|$. Hence
\begin{align*}
(b)^p &= \int_{\R^n} \int_{|h|\geq 2|x|} \frac{|u(x-h)|^p}{|h|^{n+sp}}\,dhdx\\
&\leq 
\left(\frac32\right)^{n+sp}\int_{\R^n} \int_{\frac23|y|<|x-y|<2|y|} \frac{|u(y)|^p}{|y|^{n+sp}}\,dydx\\
&= \left(\frac32\right)^{n+sp}\int_{\R^n}  \frac{|u(y)|^p}{|y|^{n+sp}}\left( \int_{\frac23|y|<|x-y|<2|y|} \,dx\right) dy\\
&\le \left(\frac32\right)^{n+sp}\int_{\R^n}  \frac{|u(y)|^p}{|y|^{n+sp}} n\omega_n 2^n|y|^n dy\\
&= n\omega_n \frac{3^{n+sp}}{2^{sp}} \int_{\R^n}  \frac{|u(y)|^p}{|y|^{sp}} dy.
\end{align*}
Observe that this last quantity is finite by Hardy's inequality.

For $(a)$, we observe that, changing variables,
\begin{align*}
(a)^p&=\int_{\R^n}\int_{|h|\geq 2|x|} m(x,h) \frac{|u(x)-u(x-h)|^p}{|h|^{n+sp}}\,dhdx\\
&= \int_{\R^n}\int_{|h|\geq 2|x-h|} m(x,h)\frac{|u(x)-u(x-h)|^p}{|h|^{n+sp}}\,dhdx:= (\tilde a)^p,
\end{align*}
where we have used the symmetry assumption \eqref{H2}. 

Observe that the sets $\{|h|\geq 2|x-h|\}$ and $\{|h|\geq 2|x|\}$ are disjoints, so
$$
2(a)^p = (a)^p + (\tilde a)^p  \leq  \iint_{\R^n\times\R^n} m(x,h) \frac{|u(x)-u(x-h)|^p}{|h|^{n+sp}}\,dhdx.
$$

Next, observe that
\begin{align*}
s J_{m,s}(u)  &\geq 2s(a)^p \ge 2s\left[ \frac{1}{(1+\ve)^{p-1}} I^p - \frac{1}{\ve^{p-1}}  (b)^p\right] \\
&\ge \frac{1}{(1+\ve)^{p-1}} \int_{\R^n} |u(x)|^p b_s(x)\,dx -  \frac{2s}{\ve^{p-1}}n\omega_n \frac{3^{n+sp}}{2^{sp}} \int_{\R^n}  \frac{|u(y)|^p}{|y|^{sp}} dy.
\end{align*}

Finally, using Fatou's Lemma, 
$$
\liminf_{s\downarrow 0} sJ_{m,s}(u) \geq \frac{1}{(1+\ve)^{p-1}} \int_{\R^n} |u(x)|^p b(x)\,dx,
$$
for any $\ve>0$ and the result follows.
\end{proof}

The next lemma gives us the upper estimate.
\begin{lema}[Limsup estimate]\label{limsupMS}
For any $u\in W_0^{s_0,p}(\R^n)$ for some $s_0\in (0,1)$, it holds that
\begin{equation}\label{limsup}
    \limsup_{s\downarrow 0}s J_{m,s}(u) \leq  \int_{\R^n} |u(x)|^p b(x)\, dx.
\end{equation}
\end{lema}

\begin{proof}
Observe that our symmetry assumption \eqref{H2} gives that
\begin{align*}
\int_{\R^n}\int_{|x-h|< |x|}  m(x,h) & \frac{|u(x)-u(x-h)|^p}{|h|^{n+sp}}\,dhdx \\
& =  \int_{\R^n}\int_{|x-h|>  |x|} m(x-h,-h)  \frac{|u(x)-u(x-h)|^p}{|h|^{n+sp}}\,dhdx \\
&=  \int_{\R^n}\int_{|x-h|>  |x|}  m(x,h)  \frac{|u(x)-u(x-h)|^p}{|h|^{n+sp}}\,dhdx.
\end{align*}
Therefore
\begin{align*}
\iint_{\R^n\times\R^n} m(x,h)&\frac{|u(x)-u(x-h)|^p}{|h|^{n+sp}}\, dhdx \\
& = 2 \int_{\R^n}\int_{|x-h|\geq |x|} m(x,h) \frac{|u(x)-u(x-h)|^p}{|h|^{n+sp}}\,dh dx.
\end{align*}
Arguing as in the previous lemma, given $\ve>0$ we get
\begin{align*}
sJ_{m,s}(u) =&  2s \left(\int_{\R^n}\int_{|x-h|\geq 2|x|}  + \int_{\R^n}\int_{|x|<|x-h|<2|x|}\right) m(x,h)\frac{|u(x)-u(x-h)|^p}{|h|^{n+sp}}\, dhdx \\
\leq &(1+\ve)^{p-1} \int_{\R^n} |u(x)|^p b_s(x)\, dx \\
& +2sm^+\left[ \left(\frac{1+\ve}{\ve}\right)^{p-1}\int_{\R^n}\int_{|x-h|\geq 2|x|} \frac{|u(x-h)|^p}{|h|^{n+sp}}\, dhdx\right.\\
&\qquad\qquad\left. + \int_{\R^n}\int_{|x|<|x-h|<2|x|} \frac{|u(x)-u(x-h)|^p}{|h|^{n+sp}}\, dhdx \right]\\
&:= (a)^p + 2sm^+ [(b)^p +(c)^p].
\end{align*}

We need to get uniform, in $s$, bounds on $(b)$ and $(c)$. For $(b)$, we use that $|x-h|\ge 2|x|$ implies that $|h|>\frac{1}{2}|x-h|$ together with Fubini's Theorem to obtain
\begin{align*}
(b)^p & \le  2^{n+sp} \int_{\R^n}\left(\int_{|x-h|\ge |x|} \frac{|u(x-h)|^p}{|x-h|^{n+sp}}\, dh\right)dx\\
&= 2^{sp}  n\omega_n \int_{\R^n}  \frac{|u(x)|^p}{|x|^{sp}}  \,dx <\infty.
\end{align*}
Hardy's inequality gives us the desired uniform bound on $(b)$ and therefore
$$
\limsup_{s\to 0^+}\; s(b)^p =  0.
$$

It remains to bound $(c)$. To begin with, we split the integral into two parts, one where $|h|$ is large and another one where $|h|$ is bounded.
\begin{align*}
(c)^p = &  \int_{\R^n}\int_{|x|<|x-h|<2|x| \atop |h|\leq N} \frac{|u(x)-u(x-h)|^p}{|h|^{n+sp}}\, dhdx \\
&+   \int_{\R^n}\int_{|x|<|x-h|<2|x| \atop |h|> N} \frac{|u(x)-u(x-h)|^p}{|h|^{n+sp}}\, dhdx\\
:= & (c_1)^p + (c_2)^p.
\end{align*}

To bound $(c_1)$ we proceed as follows
\begin{align*}
(c_1)^p &\leq  N^{p(\tau-s)} \int_{\R^n}\int_{|x|<|x-h|<2|x| \atop |h|\leq N} \frac{|u(x)-u(x-h)|^p  }{|h|^{n+\tau p}} \, dhdx\\
&\leq  N^{p(\tau-s)} [u]_{\tau, p}^p,
\end{align*}
where $\tau>s$ is fixed. From this expression, 
$$
\limsup_{s\to 0^+}\; s (c_1)^p =  0.
$$

It remains to get a bound for $(c_2)$. First, we observe that, as $|x|<|x-h|<2|x|$ and $|h|>N$, it follows that $|x|>\frac{N}{3}$ and $|x-h|>\frac{N}{3}$. Hence
\begin{align*}
(c_2)^p&\leq 
2^{p-1} \int_{\R^n}\int_{|x|<|x-h|<2|x| \atop |h|> N} \frac{|u(x)|^p}{|h|^{n+sp}}\, dhdx + 2^{p-1} \int_{\R^n}\int_{|x|<|x-h|<2|x| \atop |h|> N} \frac{|u(x-h)|^p}{|h|^{n+sp}}\, dhdx\\
&\leq  2^{p-1} \int_{\R^n}\int_{|x|>N/3 \atop |h|> N}  \frac{|u(x)|^p}{|h|^{n+sp}}\, dydx + 2^{p-1} \int_{\R^n}\int_{|x-h|>N/3 \atop |h|> N}\frac{|u(x-h)|^p}{|h|^{n+sp}}\, dhdx\\
&= 2^{p}  \int_{\R^n}\int_{|x|>N/3 \atop |h|> N}  \frac{|u(x)|^p}{|h|^{n+sp}}\, dydx\\
&= 2^{p} \int_{|x|>\frac{N}{3}} |u(x)|^p \left(n\omega_n\int_N^\infty  \frac{r^{n-1}}{r^{n+sp}}\,dr \right)dx\\
&= \frac{n\omega_n 2^{p} }{sp} \frac{1}{N^{sp}}   \int_{|x|>\frac{N}{3}} |u(x)|^p dx
\end{align*}
from where it follows that
$$
\limsup_{s\to 0+} s(c_s)^p \le \frac{n\omega_n 2^p}{p} \int_{|x|>\frac{N}{3}} |u(x)|^p dx,
$$
and this quantity is arbitrary small if $N$ is large.
\end{proof}

With the help of Lemmas \ref{liminfMS} and \ref{limsupMS} we can deduce the main result of the section.
\begin{proof}[Proof of Theorem \ref{MS}]
Immediate from Lemmas \ref{liminfMS} and \ref{limsupMS}.
\end{proof}

\section{Anisotropic nonlocal and local problems}\label{S5}

One application of the results in Section \ref{S3} is to analyze the asymptotic behavior of the solutions to anisotropic nonlocal problems. That is, given a domain $\Omega\subset \R^n$ and a source term $f\in L^{p'}(\Omega)$ one wants to analyze the limit as $s\uparrow 1$ of the solutions to
\begin{equation}\label{nonlocal}
\begin{cases}
\LL_{m,s} u_s = f & \text{in }\Omega\\
u_s=0 & \text{in } \R^n\setminus \Omega,
\end{cases}
\end{equation}
where $\LL_{m,s}$ is the Fr\'echet derivative of $J_{m,n}$ given by \eqref{Lm}.

In the case where $m=1$ this problem is well understood since the seminal works of \cite{BBM} and for some recent results regarding this problem, even in the {\em semilinear-type} case (that is when $f=f(u)$) we refer to \cite{FBS}.

In this general case, the results in Section \ref{S3}, suggest that the limit problem for \eqref{nonlocal} when $s\uparrow 1$ is 
\begin{equation}\label{local}
\begin{cases}
\LL_\A u = f & \text{in }\Omega\\
u=0 & \text{on }\partial\Omega,
\end{cases}
\end{equation}
where $\LL_\A u = -\diver(\nabla_\xi \A(x,\nabla u))$.

Recall that $\LL_\A$ is the Fr\'echet derivative of the functional
$$
J(u) = \int_{\R^n} \A(x,\nabla u)\, dx.
$$

The results of Section \ref{S3}, immediately gives:
\begin{thm}\label{gamma}
Assume that $m$ verifies \eqref{H1}--\eqref{H3} and let $\A$ be defined by \eqref{coef}. Define the functionals $J_{m,s}, J\colon L^p(\Omega)\to \bar\R$ as
\begin{align*}
&J_{m,s}(u) = \begin{cases}
\displaystyle\frac{1-s}{p}\iint_{\R^n\times\R^n} m(x,h) \frac{|u(x)-u(x-h)|^p}{|h|^{n+sp}}\, dhdx & \text{if } u\in W^{s,p}_0(\Omega)\\
\infty & \text{else}
\end{cases}\\
&J(u) = \begin{cases}
 \displaystyle\int_{\R^n} \A(x,\nabla u)\, dx & \text{if } u\in W^{1,p}_0(\Omega)\\
 \infty & \text{else.}
\end{cases}
\end{align*}
Then $J_{m,s}$ Gamma-converges to $J$ as $s\uparrow 1$.
\end{thm}

The definition and properties of Gamma-convergence can be seen in \cite{DalMaso}, The proof of Theorem \ref{gamma} is straightforward from Section \ref{S3} and the details are completely analogous as in \cite{FBS-JFA}.

The main feature of Gamma-convergence is that it implies the following result
\begin{thm}\label{minimos}
Let $J_{m,s}$ and $J$ be defined as in Theorem \ref{gamma}. Then, if $f\in L^{p'}(\Omega)$, there exists a unique minimum $u_s\in W^{s,p}_0(\Omega)$ of 
$$
J_{m,s}(v) - \int_\Omega fv\, dx
$$ 
a unique minimum $u\in W^{1,p}_0(\Omega)$ of 
$$
J(v) - \int_\Omega fv\, dx
$$ 
and $u_s\to u$ in $L^p(\Omega)$.
\end{thm}

Again, the details of Theorem \ref{minimos} with the obvious modifications, can be found in \cite{FBS-JFA}.

As a corollary of Theorem \ref{minimos} we get the connection between the solution to \eqref{nonlocal} with the solution to \eqref{local}.
\begin{cor}
For each $s\in (0,1)$, there exists a unique solution $u_s\in W^{s,p}_0(\Omega)$ to \eqref{nonlocal}. This sequence of solutions $\{u_s\}_{s\in (0,1)}$ converges, as $s\uparrow 1$, in $L^p(\Omega)$ to some function $u\in W^{1,p}_0(\Omega)$ and this function $u$ is the unique solution to \eqref{local}.
\end{cor}

\section{Remarks on homogenization}\label{S6}
The purpose of this section is to investigate the simultaneous effect that  {\em localization} (i.e. $s\uparrow 1$) and homogenization can have in some problems. 

To be precise, assume that now we have a family of kernels $m_\ve(x,h)$ satisfying \eqref{H1}--\eqref{H3}. Then, for each $\ve$ when we take $s\uparrow 1$, we obtain a limit function $\A_\ve(x,\xi)$ as defined in \eqref{coef}.

The results of the previous section tell us that the solutions $u_{s,\ve}$ of \eqref{nonlocal} with $m=m_\ve$ are converging as $s\uparrow 1$ to $u_\ve$, the solution to \eqref{local} with $\A=\A_\ve$.

The problem that we want to address is what happens when $\ve\downarrow 0$.

In order to understand this question, we focus on the model problem where $m_\ve(x,h)$ is obtain from a single kernel $m$ in the form
$$
m_\ve(x,h)= m(\tfrac{x}{\ve}, h),
$$
and $m(x,h)$ is a periodic function in $x$ of period 1 in each $x_i$, $i=1,\dots, n$.

To keep things even simpler, we start with the one dimensional problem.

\subsection{The one-dimensional case}
In the $1-$dimensional case explicit formulas describing the behavior of the limit problems are available. Indeed, consider a $1-$periodic function in $x$, $m(x,h)$ satisfying \eqref{H1}--\eqref{H3}. According to  Theorem \ref{teo.bbm} we have that
$$
\lim_{s\uparrow 1} (1-s) \iint_{\R\times \R} m(\tfrac{x}{\ve},h) \frac{|u(x)-u(x-h)|^p}{|h|^{1+sp}}\,dxdh =  \int_{\R^n} \A(\tfrac{x}{\ve}, |u'(x)|) \,dx =: J^\ve(u).
$$
In this one-dimensional case, the limit function $\A(y,\xi)$ can be easily computed as
$$
\A(y, \xi)=\frac1p \left(a(y,-1) +a(y,1)\right) |\xi|^p =: A(y) |\xi|^p.
$$
When $\ve$ vanishes, it is well known that (see \cite[Proposition 3.7]{FBPS}) $J^\ve$ Gamma-converges to $J^*$, where
$$
J^*(u) = \int_\R A^* |u'(x)|^p\, dx,
$$
and $A^*$ is a constant coefficient given by
$$
A^*:=\left(\int_0^1 A(t)^{-1/(p-1)}\, dt\right)^{-1/(p-1)}.
$$
This fact gives as a result that minimizers $u_{s,\ve}$ of \eqref{nonlocal} with coefficients $m_\ve$ verifiy that
\begin{align*}
\lim_{\ve\downarrow 0} \left( \lim_{s\uparrow 1} u_{s,\ve} \right) = u^*,
\end{align*}
where $u^*$ is the solution to
\begin{equation}\label{homog}
\begin{cases}
-(A^* |u'|^{p-2}u')' = f & \text{in }\Omega\\
u=0 & \text{on }\partial\Omega.
\end{cases}
\end{equation}

On the other hand, it is well-known (see, for instance \cite{FBRS}) that as $\ve \downarrow 0$ it holds that solutions $u_{s,\ve}$ of \eqref{nonlocal} with coefficients $m_\ve$ converge to the solution $\bar u_s$ of \eqref{nonlocal} with coefficient $\bar m(h)$ given by
$$
\bar m(h) =\int_0^1 m(t,h)\,dt.
$$

Finally, applying Theorem \ref{teo.bbm} we arrive at
\begin{align*}
\lim_{s\uparrow 1} \left(  \lim_{\ve\downarrow 0}  u_{s,\ve} \right) =   \bar u,
\end{align*}
where $\bar u$ is the solution to the problem
\begin{equation}\label{homog2}
\begin{cases}
-(\bar A |u'|^{p-2}u')' = f & \text{in }\Omega\\
u=0 & \text{on }\partial\Omega.
\end{cases}
\end{equation}
In this case, $\bar A$ is given by
$$
\bar A= \frac1p(\bar m (-1)+\bar m(1)) = \frac1p \int_0^1 m(t,-1)+m(t,1)\,dt.
$$

From these simple formulas one can immediately see that the localization process and the homogenization process are not interchangeables.

\subsection{The general case}
The computations of the previous subsection can be extended with some care to the $n-$dimensional case.

Let now $m(x,h)$ be a $Q-$periodic function in its first variable, $Q$ being the unit cube in $\R^n$, and  satisfying hypothesis \eqref{H1}--\eqref{H3}.

Let us now state the problem in a precise way.

A function $u_{s, \ve}\in W^{s,p}_0(\Omega)$ is a weak solution of 
$$
\begin{cases}
\LL_{m_\ve, s} u = f & \mbox{in }\Omega\\
u=0 & \mbox{in }\R^n\setminus\Omega,
\end{cases}
$$
if
$$
\langle \LL_{m_\ve, s} u_{s,\ve}, v\rangle = \int_\Omega fv\ dx,
$$
for every $v\in W^{s,p}_0(\Omega)$ where $f\in W^{-s,p'}(\Omega)$ and  $\LL_{m_\ve, s}$ is the Fr\'echet derivative of $J_{m_\ve, s}$ given in \eqref{Lm}. Observe that it follows that $\langle \LL_{m_\ve, s} u, v\rangle$ is given by
$$
(1-s)\iint_{\R^n\times\R^n} m(\tfrac{x}{\ve}, h) \frac{|u(x)-u(x-h)|^{p-2}(u(x)-u(x-h))(v(x)-v(x-h))}{|h|^{n+sp}}\, dhdx.
$$

Our first result concerns with the problem of first localizing and then homogenizing, that is, we first take the limit as $s\uparrow 1$ and then the limit as $\ve\downarrow 0$.

\begin{prop}\label{lim.ve.s}
It holds that
$$
\lim_{\ve\downarrow 0} \left( \lim_{s\uparrow 1} u_{s,\ve} \right)= u^*,
$$
(in the $L^p(\Omega)$ sense) where $u^*\in W^{1,p}_0(\Omega)$ is the weak solution of
$$
\begin{cases}
-\diver(\nabla_\xi\A^*(\nabla u^*)) = f & \text{in }\Omega\\
u^*=0 & \text{on }\partial\Omega,
\end{cases}
$$
and
$$
\A^*(\xi)= \inf_{v\in W^{1,1}_\text{per}(Q)} \int_Q \A(y, \xi+\nabla v(y))\, dy
$$
\end{prop}

\begin{proof}
First, we have to take the limit as $s\uparrow 1$ for fixed $\ve>0$, but this was carried out in Section \ref{S5}, and it holds that
$$\lim_{s\uparrow 1} u_{s,\ve} = u_\ve,$$
where $u_\ve\in W^{1,p}_0(\Omega)$ is the solution to
$$
\begin{cases}
-\diver(\nabla_\xi \A_\ve(x,\nabla u_\ve))= f& \text{in }\Omega\\
u_\ve=0 & \text{on }\partial\Omega.
\end{cases}
$$
Now, we can apply the results of \cite[Chapter 24]{DalMaso} to conclude the desired result as $\ve\downarrow 0$. 
\end{proof}

To finish the section, we now deal with the case where first we homogenize and then localize, i.e. first take the limit $\ve\downarrow 0$ and then the limit $s\uparrow 1$.

\begin{prop}\label{lim.s.ve}
It holds that
$$
\lim_{s\uparrow 1} \left(  \lim_{\ve\downarrow 0}  u_{s,\ve} \right) = \bar u
$$
(in the $L^p(\Omega)$ sense), where $\bar u\in W^{1,p}_0(\Omega)$ is the solution to
$$
\begin{cases}
-\diver(\nabla_\xi \bar\A(\nabla \bar u))= f & \text{in }\Omega\\
\bar u = 0 & \text{on }\partial\Omega,
\end{cases}
$$
where $\bar \A$ is given by
$$
\bar \A(\xi) = \int_Q \A(y,\xi)\, dy,
$$
and $\A$ is the one given by \eqref{coef}.
\end{prop}

\begin{proof}
We first have to take the limit as $\ve\downarrow 0$ for fixed $s\in (0,1)$. But this problem was already solved in \cite{FBRS} and what is known is that
$$
\lim_{\ve\to 0} u_{s,\ve} = \bar u_s,
$$
where $\bar u_s$ is the solution to
$$
\begin{cases}
\LL_{\bar m, s} \bar u_s = f & \text{in }\Omega\\
\bar u_s = 0&\text{in }\R^n\setminus\Omega,
\end{cases}
$$
and 
$$
\bar m(h) = \int_Q m(y,h)\, dy.
$$
Now, we can take the limit as $s\uparrow 1$ using the results of Section \ref{S5} to conclude the desired result.
\end{proof}

\bibliography{biblio}
\bibliographystyle{plain}

\end{document}